\documentclass{amsart}
\usepackage{amssymb, thmtools}

\usepackage{graphicx}
\usepackage{amscd}
\usepackage{scalerel}
\usepackage{stackengine,wasysym}
\usepackage{tikz-cd}

\newtheorem{theorem}{Theorem}[section]
\theoremstyle{plain}

\newtheorem{corollary}{Corollary}[section]

\newtheorem{definition}{Definition}[section]

\newtheorem{lemma}{Lemma}[section]

\newtheorem{proposition}{Proposition}[section]

\numberwithin{equation}{section}

\newtheorem*{theorem*}{Theorem}

\makeatletter
\thmt@define@thmuse@key{tag}{%
  \thmt@suspendcounter{\thmt@envname}{#1}%
}
\makeatother

\newcommand{\R}{\mathbb{R}}
\newcommand{\Sp}{\mathrm{Span}}
\newcommand{\Sec}{\mathrm{Sec}}

\newcommand{\is}{\mathfrak{i}}
\newcommand{\iso}{\mathcal{I}}

\newcommand{\poch[1]}{\frac{\partial}{\partial #1}}

\begin{document}

\title[Isometries]{ON THE EXISTENCE AND PROLONGATION OF INFINITESIMAL ISOMETRIES ON SPECIAL SUB-RIEMANNIAN MANIFOLDS}

\author[M. Grochowski]{Marek Grochowski}
\address{Faculty of Mathematics and Natural Sciences\\ Cardinal Stefan Wyszy\'nski University in Warsaw\\
ul. Dewajtis 5, 01-815 Warszawa, Poland}
\email{m.grochowski@uksw.edu.pl}
\date{\today}

\begin{abstract}
In the present paper we deal with (local) infinitesimal isometries of special sub-Riemannian manifolds (a contact and oriented sub-Riemannian manifold is called special if the Reeb vector field is an isometry).
The objective of the paper is to find some conditions on  such manifolds which allow one to construct, locally around a given point, infinitesimal isometries and then, possibly, to prolong them onto bigger domains. The mentioned conditions are related to the so-called $\is^*$-regular and $\is$-regular points, the notions introduced by Nomizu in \cite{Nom} in the Riemannian setting and slightly modified by the author.
\end{abstract}

\maketitle

\section{Introduction and statement of results}

A sub-Riemannian structure on a smooth (i.e., $C^{\infty}$) manifold $M$ is a smooth distribution $H\subset TM$ endowed with a smooth Riemannian metric $g$ on $H$. It is assumed that $H$ has constant rank and is bracket generating. This last condition means that the vector fields tangent to $H$ together with their Lie brackets of any order span the whole tangent bundle. The triple $(M,H,g)$ is called a \textit{sub-Riemannian manifold}. In the sequel we will deal with a particular case of sub-Riemannian manifolds, namely we will assume that $\dim M = 2n+1$ and the distribution $H$ is contact. In such a case we speak about contact sub-Riemannian manifolds. Sub-Riemannian geometry has been intensively studied for many years (see e.g. \cite{Agr,AgrIsomp,AgrBar,BoscAgr,AMS,Bell, POP, GroMaVas,LuiSuss} and the reference sections therein) and various aspects of the theory have been investigated, like for instance properties of the distance function, geodesics, conjugate and cut loci, nilpotent approximations, existence and properties of isometries, smooth volumes canonically associated with sub-Riemannian structures and many others.

The main object of interest of the present paper are infinitesimal isometries. Fix a sub-Riemannian manifold $(M,H,g)$. A diffeomorphism $\varphi:M\longrightarrow M$ is an isometry if $d\varphi(H)\subset H$ and $g(d_q\varphi(v),d_q\varphi(w)) = g(v,w)$ for every $q\in M$, $v,w\in H_q$. A vector field $Z$ is called an infinitesimal isometry if its (local) flow is composed of isometries. It is proved that the group of isometries is a Lie group whose maximal dimension in a contact case is $(n+1)^2$ (where $\dim M = 2n+1$). Sub-Riemannian isometries are treated for instance in \cite{Strich, leDonne,GrKrESAIM, GroWar, GroAnn}.

We will investigate isometries of \textit{special} sub-Riemannian manifolds. Let us explain what this latter notion means. Suppose that $(M,H,g)$ is a contact sub-Riemannian manifold which is oriented as a contact manifold. Under such an assumption the Reeb vector field (denoted by $\xi$) associated to the structure $(H,g)$ is defined globally on $M$ -- see the definitions below. We say that $(M,H,g)$ is special if $\xi$ is an infinitesimal isometry. The objective of the paper is to find some conditions on special sub-Riemannian manifolds which allow one to construct, locally around a given point, infinitesimal isometries, and then to prolong them onto bigger domains. Take a point $q\in M$. By $\mathfrak{i}^*(q)$ we will denote the Lie algebra of germs at $q$ of local infinitesimal isometries of $(M,H,g)$. A point $q$ is called $\is^*$-regular if the function $p\longrightarrow\dim\is^*(p)$ is constant on a neighborhood of $q$. The set of $\is^*$-regular points is denoted by $M^*$. It is proved  that $M^*$ is open and dense in $M$. Remark that in the paper \cite{GroAnn} all possible dimensions of $\mathfrak{i}^*(q)$, where $q\in M^*$, were computed in case $\dim M = 3$. If $Z$ is a vector field then $(Z)_q$ will stand for its germ at a point $q$. We will prove that if $q\in M^*$ has a simply connected neighborhood $U\subset M^*$ then each $Z^* \in\is^*(q)$ has a unique extension to an element of $\iso(U)$, where the latter denotes the Lie algebra of infinitesimal isometries defined on $U$. More precisely we will prove

\begin{theorem}[tag = 1] \label{th1}
Suppose that $(M,H,g)$ is a simply connected special sub-Riemannian manifold such that every point of $M$ is $\is^*$-regular. Then for every point $q\in M$ and every $Z^*\in\is^*(q)$ there exists a unique infinitesimal isometry $Z$ such that $(Z)_q = Z^*$. In particular, there exists a Lie algebra isomorphism $\iso(M)\longrightarrow\is^*(q)$ for every $q\in M$.
\end{theorem}

Next we introduce the notion of $\is$-regular points (the idea of $\is^*$-regular and $\is$-regular points is taken, after a suitable adaptation, from \cite{Nom}). For a point $q\in M$ let $\mathfrak{a}(q) = H_q\oplus E_q\oplus\R$, where $E_q$ is the set of endomorphism $H_q\longrightarrow H_q$ which are skew-symmetric with respect to $g$. 
Let us denote by $\is(q)$ the set that consists of all $(X,A,c)\in\mathfrak{a}(q)$ which satisfy $(\nabla_{X+c\xi(q)} + A)(\nabla^iR) = 0$ and $(\nabla_{X+c\xi(q)} + A)(\nabla^id\alpha) = 0$ for $i=0,1,\dots$ Here $\nabla$ is the canonical sub-Riemannian connection, $R$ is the curvature tensor of $\nabla$ and $\alpha$ is the normalized contact form -- see Sections \ref{SecReebIsm}, \ref{sRCON} below. 
A point $q$ is said to be $\is$-regular if the function $p\longrightarrow\dim\is(p)$ is constant in a neighborhood of $q$. The set of such points will be denoted by $M^o$. We prove that $M^o$ is open and dense in $M$. It is clear that local behavior of infinitesimal isometries, their existence and possibility of their extension are determined by properties of the sets $\is^*(q)$ and $\is(q)$. Note that we have a natural injection $\is^*(q)\ni Z^*\longrightarrow (PZ(q),A_Z(q),f_Z(q))\in\is(q)$, where $Z$ is a representative of $Z^*$ and $Z = PZ + f_Z\xi$, $PZ$ being the horizontal part of $Z$ -- see Section \ref{SecAnalytic} for details.
Further we prove that starting from any $(X,A,c)\in\is(q)$, $q\in M^o$, we can build a unique infinitesimal isometry $Z$ defined in a neighborhood of $q$ such that $Z(q) = X + c\xi(q)$. More exactly we obtain 
\begin{theorem}[tag = 2]\label{th2} 
 Suppose that $q\in M$ is an $\is$-regular point. Then the injection $\is^*(q_0)\longrightarrow\is(q_0)$ is onto.
\end{theorem}

Theorem \ref{th2} is perhaps the most important result of the paper.

In the analytic case we prove 
\begin{theorem}[tag = 3]\label{th3}
 Suppose that $(M,H,g)$ is an analytic special sub-Riemannian manifold. Then every point $q\in M$ is $\is$-regular, i.e., $M = M^o$.
\end{theorem}
We should note here that when $M$ as well as $H$ and $g$ are supposed to be analytic then it can be proved that the canonical sub-Riemannian connection is also analytic.

Combining the above results one immediately obtains

\begin{theorem}[tag = 4]\label{th4}
 Suppose that $(M,H,g)$ is an analytic special sub-Riemannian manifold. Then $M = M^o = M^*$. In particular every point $q\in M$ is $\is^*$-regular.
\end{theorem}

Note that if $M$ is connected the statement that every point in $M$ is $\is^*$-regular is equivalent to the claim that the function $M\ni q\longrightarrow\dim\is^*(q)$ is constant.
Now, as a conclusion we can make the following observation: 
\begin{corollary}\label{corIntr}
 If $(M,H,g)$ is an analytic simply connected special sub-Riemannian manifold then for every point $q\in M$ and every germ $Z^*\in\is^*(q)$ there exists a unique infinitesimal isometry $Z\in\iso(M)$ such that $(Z)_q = Z^*$.
 \end{corollary}

It should be remarked that this last result can also be obtained as a consequence of \cite[Theorem 1]{Nom} treating the Riemannian case.\smallskip

The methods used in the proofs of the above theorems are modelled on the paper \cite{Nom} by Nomizu, where similar results in the Riemannian case were proven. Our main tool is a canonical sub-Riemannian connection introduced in \cite{GroConn,GroConnCor} (see Section \ref{sRCON} for more details). It turns out that such a connection perfectly fits to the type of problems considered in \cite{Nom}, and makes it possible to perform computations similar to those done in Riemannian geometry by use of the Levi-Civita connection. What we have to do is to adapt some notions and definition to the sub-Riemannian setting and prove several preparatory and auxiliary results. Some of them look similarly as their Riemannian counterparts, some are different.

\bigskip
\noindent \textbf{Content of the paper.}
In Section \ref{SecReebIsm} we state basis properties of the Reeb vector field and infinitesimal isometries. In Section \ref{sRCON} we recall the construction of a canonical torsion-free sub-Riemannian connection that was introduced in \cite{GroConn,GroConnCor}. Such a connection exists on special sub-Riemannian manifolds. Some results, especially concerning the curvature, are new in comparison with \cite{GroConn,GroConnCor}. At the beginning of Section \ref{SecProl} we associate to every infinitesimal isometry a certain skew-symmetric operator, analogously as it is done in the Riemannian case. Further, using it,  we describe the prolongation procedure of infinitesimal isometries along (piecewise) smooth curves. We show that such a prolongation always exists along curves which entirely consist of $\is^*$-regular points and the prolongation is unique. As a corollary of these considerations we prove Theorem \ref{th1}. In Section \ref{SecAnalytic} we introduce the notion of $\is$-regular points. First we show that under the assumption of analycity of $(M,H,g)$ every point in $M$ is $\is$-regular which proves Theorem \ref{th3}. Next we show how to extend each element of $\is(q)$, $q\in M^o$, to an infinitesimal isometry defined around $q$ thus proving Theorem \ref{th2}. Having proved the above results, Theorem \ref{th4} is a simple consequence of thereof.

\bigskip
\noindent \textbf{Convention and notation.}
In the paper all manifolds, vector fields etc. are supposed to be smooth, i.e., of class $C^{\infty}$, unless analycity is assumed. If $E\longrightarrow M$ is a vector bundle and $U\subset M$ is an open subset, then by $\Sec(U,E)$ we denote the $C^{\infty}(U)$-module of sections of $E$ defined on $U$. Also we write $\Sec(E)$ for $\Sec(M,E)$. If $(M,H,g)$ is a sub-Riemannian manifold then every vector field $X\in\Sec(H)$ will be referred to as a \textit{horizontal vector field}. A piecewise smooth curve $\gamma:[a,b]\longrightarrow M$ such that $\dot{\gamma}(t)\in H_{\gamma(t)}$ for almost every $t$ will be said to be a \textit{horizontal curve}.

\section{Reeb vector field and Infinitesimal isometries}\label{SecReebIsm}
Let $(M,H,g)$ be a smooth contact sub-Riemannian manifold of dimension $2n+1$. We will suppose it to be oriented as a contact manifold, i.e., the bundles $TM$ and $H$ are oriented. Equivalently we can express this by saying that there exists a globally defined contact $1$-form $\alpha$ such that $H = \ker\alpha$. We normalize $\alpha$ in the following way. Let $X_1,\dots,X_{2n}$ be a local positively oriented orthonormal basis of $H$. Then we assume 
\begin{align}\label{normaliz}
 \bigwedge\nolimits ^n d\alpha(X_1,\dots,X_{2n}) = 1\text{;}
\end{align}
note that the above equality does not depend on the choice of an orthonormal frame. If $n$ is even, there exist two such forms that differ by sign and we choose either of them. From now on we fix a contact form $\alpha$, $\ker\alpha = H$, which satisfies (\ref{normaliz}). Such $\alpha$ will be referred to as the \textit{normalized contact form}.    

Having fixed $\alpha$ we define the Reeb vector field $\xi$ as the solution to the system of equations: $d\alpha(\xi,\cdot)= 0$, $\alpha(\xi) = 1$. The field $\xi$ has many special properties which will be stated in detail in the proposition below.

A diffeomorphism $\varphi\colon M\longrightarrow M$ is called an \textit{isometry} of $(M,H,g)$ if (i) $d\varphi(H)\subset H$ and, moreover, (ii) $g(d_q\varphi(v),d_q\varphi(w)) = g(v,w)$ for every $q\in M$ and $v,w\in H_q$. It is clear that if $\varphi$ is an isometry then $\varphi^*\alpha = \pm\alpha$ as well as $\varphi_*\xi = \pm\xi$, depending on whether $\varphi$ preserves or reverses the given orientation. A mapping $\varphi\colon U\longrightarrow M$, where $U\subset M$ is an open subset, which a diffeomorphism onto its image $\varphi(U)$ and satisfies (i), (ii) for all $q\in U$ and $v,w\in H_q$ will be referred to as a \textit{local isometry}. Isometries map minimizing sub-Riemannian geodesics to minimizing sub-Riemannian geodesics and preserve the sub-Riemannian distance induced by the structure $(H,g)$. Note in this place that a mapping $M\longrightarrow M$ that preserves the sub-Riemannian distance is automatically smooth, as it is proved in \cite{leDonne}.

A vector field $Z\in\Sec(TM)$ is called a \textit{contact vector field} if $[Z,X]\in\Sec(U,H)$ for every open $U\subset M$ and $X\in\Sec(U,H)$. Equivalently we can say that $Z$ is contact if around any point its (local) flow $\varphi^t$ satisfies $d_q\varphi^t(H_q)\subset H_{\varphi^tq}$ for all $t$ and $q$ for which it makes sens. Below we will need
\begin{proposition}\label{ContCond}
 $Z\in\Sec(TM)$ is a contact vector field if and only if 
 \[
  d\alpha(Z,X) = -X(\alpha(Z))
 \]
 for every open subset $U\subset M$ and $X\in\Sec(U,H)$.
\end{proposition}

A vector field $Z\in\Sec(TM)$ is called an \textit{infinitesimal isometry} or a \textit{Killing vector field} if its (local) flow is composed of (local) isometries. From the very definition it follows that
\begin{proposition}\label{InfIsmCon}
 $Z\in\Sec(TM)$ is an infinitesimal isometry if and only if $(i)$ $[Z,X]\in\Sec(U,H)$ and $(ii)$ $Z(g(X,Y)) = g([Z,X],Y) + g(X,[Z,Y])$ for every open $U\subset M$ and for all $X,Y\in\Sec(U,H)$. 
 \qed
\end{proposition}
Basic properties of the Reeb field and infinitesimal isometries are collected in the following proposition.
The algebra of infinitesimal isometries defined on an open subset $U\subset M$ will be denoted by $\iso(U)$. As it was mentioned, $\dim\iso(U) \leq (n+1)^2$.
\begin{proposition}\label{PropIsom}
Let $U$ be an open subset of $M$. Then
 \begin{enumerate}
  \item\label{a1} If $Z\in\iso(U)$ then $[\xi,Z]=0$, i.e., the Reeb field commutes with infinitesimal isometries;
  \item\label{a2} $\mathcal{L}_Z\alpha = 0$ for all $Z\in\iso(U)$;
  \item\label{a3} If $Z\in\Sec(TM)$ is an infinitesimal isometry on an open and dense subset of $M$, then $Z\in\iso(M)$.
  \item\label{a4} If $Z_1,Z_2\in\iso(U)$, where $U$ is assumed to be connected, and $Z_{1|V} = Z_{2|V}$ for an open set $V\subset U$ then $Z_1 = Z_2$;
  \item\label{a5} For every non-trivial local isometry $Z\in\iso(U)$, where $U$ is assumed to be connected, the set $\{q\in U:\;Z(q)\notin H_q\}$ is open and dense in $U$;
  \item\label{a6} If $Z\in\iso(U)$ and $hZ\in\iso(U)$, where $h\in C^{\infty}(U)$, then $h$ is constant;
  \item\label{a7} If $h\xi\in\iso(U)$, where $h\in C^{\infty}(U)$, then $h$ is constant;
  \item\label{a8} Suppose that $(M,H,g)$ is analytic and connected and let $Z\in\Sec(TM)$ be an analytic vector field. If $Z$ is a contact vector field (resp., infinitesimal isometry) on an open set $U\subset M$ then $Z$ is a contact vector field (resp., infinitesimal isometry) on $M$.
 \end{enumerate}
\end{proposition}
\begin{proof}
(\ref{a1}) -- (\ref{a7}) are proved in \cite{GroAnn}. We will prove (\ref{a8}). Fix a point $q_0\in U$. Take an arbitrary $q_1\in M$ and an analytic curve $\gamma\colon[0,1]\longrightarrow M$ such that $\gamma(0) = q_0$, $\gamma(1) = q_1$. Cover $\gamma([0,1])$ by a finite family $U_0,\dots,U_k$ of open sets such that $q_0\in U_0$, $q_1\in U_k$, $U_i\cap U_{i+1}\ne\emptyset$, $i = 0,\dots k$, and for each $i$ there exist linearly independent $X^{(i)}_1,\dots,X^{(i)}_{2n}\in\Sec(U_i,H)$ which are supposed to be analytic.

Let us assume that for $i\geq 0$ 
\[
 \alpha([Z,X^{(i)}_j]) = 0
\] 
on $U_i$, $j=1,\dots,2n$. On $U_i\cap U_{i+1}$ we have
\[
 X^{(i+1)}_j = f^k_jX^{(i)}_k
\]
for some smooth functions $f^k_j$; here and below we use the summation convention with indices varying from $1$ to $2n$. Then 
\[
 \alpha([Z,X^{(i+1)}_j]) = f^k_j\alpha([Z,X^{(i)}_k]) = 0
\]
which means that $\alpha([Z,X^{(i+1)}_j])$, $j=1,\dots,2n$, are functions analytic on $U_{i+1}$ and vanishing on $U_i\cap U_{i+1}$. Therefore $\alpha([Z,X^{(i+1)}_j]) = 0$ on $U_{i+1}$ for all $j$ which proves that $Z$ is a contact vector field on $M$. The statement concerning being an infinitesimal isometry is proved analogously.
\end{proof}

Now we introduce the notion of the so-called $\is^*$-regular points which will be needed in the sequel. To this end, fix a point $q\in M$ and consider a sequence of connected neighborhoods $\{U_m\}$ of $q$ such that 
\begin{align*}
 \overline{U}_{m+1}\subset U_m,\; m=1,2\dots, \;  \bigcap_{m=1}^{\infty}U_m = \{q\} \text{.}
\end{align*}
It follows from the proposition above that the mapping $\iso(U_m)\longrightarrow\iso(U_{m+1})$, $Z\longrightarrow Z_{|U_{m+1}}$ is injective. Consequently, the sequence $\{\dim\iso(U_m)\}$ of integers is non-decreasing and therefore stabilizes from a certain moment on a number $N\leq (n+1)^2$. Denote by $\is^*(q)$ the algebra of germs at $q$ of local infinitesimal isometries of $(M,H,g)$. Thus there exists an integer $m_0$ (depending on $q$) such that $\dim\is^*(q) = \dim\iso(U_m)$ for every $m>m_0$.
In particular we conclude that every point $q$ has a neighborhood $U$, which can be chosen to be arbitrarily small, such that the mapping $\iso(U)\ni Z\longrightarrow (Z)_q\in\is^*(q)$ is an isomorphism; here and below $(Z)_q$ stands for the germ of $Z$ at $q$. Such a neighborhood $U$ of $q$ will be called an \textit{$\iso$-special neighborhood of $q$}. A point $q\in M$ will be said to be \textit{$\is^*$-regular} if the function $p\longrightarrow\dim\is^*(p)$ is constant in a neighborhood of $q$. Denote by $M^*$ the set of $\is^*$-regular points of $M$.
 
\begin{proposition}\label{M^*reg}
 The set $M^*$ is open and dense in $M$.
\end{proposition}
\begin{proof}
 The proof is similar to the proof of Proposition \ref{OpenDens} below.
\end{proof}

\begin{proposition}\label{prop01}
 Suppose that $q$ is a $\is^*$-regular point and let $U$ be its $\iso$-special neighborhood chosen in such a way that the function $U\ni p\longrightarrow\dim\is^*(p)$ is constant. Then for every $p\in U$ and every $X^*\in\is^*(p)$ there exists $X\in\is(U)$ with $(X)_p = X^*$.
\end{proposition}
\begin{proof}
 Indeed, take $p\in U$. Clearly, $\dim\is^*(p) = \dim\is^*(q) = \dim\is(U) $ and $U\ni Z\longrightarrow (Z)_p\in \is^*(p)$ is a linear injection.
\end{proof}

\section{sub-Riemannian Connection}\label{sRCON}
In this section we present the construction of a canonical sub-Riemannian connection which was introduced in \cite{GroConn, GroConnCor}. We state some new results which are not present in \cite{GroConn, GroConnCor}. 
Again, $(M,H,g)$, $\dim M = 2n+1$, is a fixed contact sub-Riemannian manifold, $\alpha$ is the normalized contact form and $\xi$ is the Reeb vector field. By $L_H(M)$ we denote the \textit{bundle of horizontal frames}, i.e. the space
\begin{align*}
 L_H(M) = \{(q;v_1,\dots,v_{2n}):\; q\in M, \;v_1,\dots,v_{2n} \text{ is a basis of } H_q\} \text{.}
\end{align*}
There is a natural right action of the general linear group $GL(2n)$ on $L_{H}(M)$: for $a\in GL(2n)$ and $l = (q;v_1,\dots,v_{2n})\in L_{H}(M)$
\begin{align}\label{action}
 R_al = (q;a_1^iv_i,\dots,a_{2n}^iv_i) \text{;}
\end{align}
Clearly, $\pi:L_H(M)\longrightarrow M$, where $\pi(q;v_1,\dots,v_{2n}) = q$, is a principle bundle with structure group $GL(2n)$. Note that if $l = (q;v_1,\dots,v_{2n})\in L_H(M)$ then $l$ can be regarded as a linear isomorphism $l:\R^{2n}\longrightarrow H_{\pi(l)}$, $l(\zeta) = \zeta^iv_i$. 

By
\begin{align}\label{project}
 P:TM\longrightarrow H
\end{align}
we denote the projection determined by the splitting $TM = H\oplus \Sp\{\xi\}$. We define a sub-Riemannian version of the canonical $1$-form from the linear frame bundles theory. Let $\theta$ be a $1$-form on $L_H(M)$ with values in $\R^{2n}$ defined by formula
\begin{align}\label{CanForm}
 \theta_l = l^{-1}\circ P\circ d_l\pi:T_lL_H(M)\longrightarrow \R^{2n}\text{,}
\end{align}
$l\in L_H(M)$. $\theta$ will be referred to as the \textit{canonical $1$-form}. Clearly, $R^*_a\theta = a^{-1}\cdot\theta$, $a\in GL(2n)$. 

Let us pick a connection $\Gamma$ on the bundle $L_H(M)$. In other words $\Gamma$ is a distribution $\Gamma\subset TL_H(M)$ such that  
\begin{align}\label{split}
 TL_H(M) = \Gamma\oplus V\text{,}
\end{align}
with $V = \ker d\pi$, and $R_{a*}(\Gamma) = \Gamma$ for every $a\in GL(2n)$. In our case $\Gamma$ splits into the sum
\begin{align*}
 \Gamma = \Gamma^H\oplus\Gamma^{\xi}\text{,}
\end{align*}
where $\Gamma^H = (d\pi_{|\Gamma})^{-1}(H)$, $\Gamma^{\xi} = (d\pi_{|\Gamma})^{-1}(\Sp\{\xi\})$.
The bundle $H\longrightarrow M$ is a vector bundle associated with $L_H(M)$ with typical fiber $\R^{2n}$, so in the classical manner \cite{KoNo} $\Gamma$ defines the covariant differentiation operator
\begin{align*}
 \nabla:\Sec(TM)\times \Sec(H) \longrightarrow \Sec(H) \text{:}
\end{align*}
if $q\in M$, and $Z\in\Sec(TM)$, $X\in\Sec(H)$ then, by definition,
\begin{align} \label{CovDiff}
 (\nabla_ZX)(q) = l(Z^*_l(F_X)) \text{,}
\end{align}
where $l\in\pi^{-1}(q)$, $Z^*$ is the $\Gamma$-horizontal lift of $Z$, i.e., the vector field on $L_H(M)$ with values in $\Gamma$ such that $d\pi(Z^*) = Z$, and $F_X:L_H(M)\longrightarrow\R^{2n}$ is a smooth function defined by $F_X(p) = p^{-1}(X(\pi(p)))$. Of course the value of (\ref{CovDiff}) does not depend on the choice of $l\in\pi^{-1}(q)$. Note that $\nabla$ can be equivalently defined using the notion of parallel transport. Given a connection $\Gamma$ on $L_H(M)$ we defined its \textit{torsion form} $\Theta$ as
\begin{align*}
 \Theta = d\theta\circ (p_{\Gamma},p_{\Gamma})\text{,}
\end{align*}
where $p_{\Gamma}:TL_H(M)\longrightarrow\Gamma$ is the projection corresponding to the splitting (\ref{split}). $\Theta$ is a $2$-form on $L_H(M)$ with values in $\R^{2n}$. Writing $\Theta = (\Theta^1,\dots,\Theta^{2n})$, $\theta = (\theta^1,\dots,\theta^{2n})$ and $\eta = \pi^*(\alpha)$ we have \cite{GroConn}
\begin{align*}
 \Theta^i = T^i_{jk}\theta^j\wedge\theta^k + S^i_j\theta^j\wedge\eta
\end{align*}
for some smooth functions $T^i_{jk}$ and $S^i_j$. The expression $T^i_{jk}\theta^j\wedge\theta^k$ is called the \textit{horizontal torsion}, while $S^i_j\theta^j\wedge\eta$ is the \textit{vertical torsion}. The torsion form determines the torsion tensor $T:\Sec(TM)\times\Sec(TM)\longrightarrow\Sec(H)$ by formula
\begin{align}\label{torformula}
 T(X,Y)(q) = l(\Theta_l(X^*,Y^*)) \text{,}
\end{align}
$l\in\pi^{-1}(q)$ and $X^*$ (resp. $Y^*$) is a $\Gamma$-horizontal lift of $X$ (resp. $Y$).
It can be shown that for all $X,Y \in\Sec(H)$, $Z\in\Sec(\Sp\{\xi\})$
\begin{align*}
 T(X,Y) = \nabla_XY - \nabla_YX - P([X,Y]) \text{,}
\end{align*}
and
\begin{align}\label{VerTor}
 T(Z,X) = \nabla_ZX - P([Z,X]) \text{.}
\end{align}
It follows that the horizontal torsion (resp. vertical torsion) vanishes if and only if $\nabla_XY = \nabla_YX - P([X,Y])$ (resp. $\nabla_{\xi}X = [\xi,X]$) for all $X,Y\in\Sec(H)$.

Now let us consider the metric reduction of the bundle $L_H(M)$, i.e., the bundle of orthonormal horizontal frames
\begin{gather*}
 O_{H,g}(M) =  \\
 \{(q;v_1,\dots,v_{2n}):\; q\in M \text{ and } v_1,\dots,v_{2n} \text{ is an orthonormal basis of } H_q\}\text{.}
\end{gather*}
This is a principle bundle with structure group $O(2n)$, where the action is given by (\ref{action}). Take a connection $\Gamma$ on $O_{H,g}(M)$. Such a connection is metric in the sense that if $\nabla$ stands for the covariant differentiation induced by $\Gamma$ then
\begin{align}\label{MetrCon}
 Z(g(X,Y)) = g(\nabla_ZX,Y) + g(X,\nabla_ZY)
\end{align}
for every $X,Y\in\Sec(H)$, $Z\in\Sec(TM)$. Connections on the bundle $O_{H,g}(M)$ will be referred to as \textit{sub-Riemannian connections}.

Suppose now that $Z$ is an infinitesimal isometry. $Z$ naturally lifts to a vector field $\hat{Z}$ on $O_{H,g}(M)$ in the following manner. Fix an arbitrary point $q_0\in M$. Let $\psi^t:U\longrightarrow M$ be the (local) flow of $Z$ defined in a neighborhood $U$ of $q_0$. Then the formula 
\begin{align}\label{FlowLift}
 \hat{\psi}^t(q;v_1,\dots,v_{2n}) = (\psi^t(q);d_q\psi^t(v_1),\dots,d_q\psi^t(v_{2n}))
\end{align}
defines the flow $\hat{\psi}^t: \pi^{-1}(U)\longrightarrow O_{H,g}(M)$, and we set 
\begin{align*}
 \hat{Z}(l) = \frac{d}{dt}\Big\vert_{t=0}\hat{\psi}^t(l) \text{.} 
\end{align*}
Such a field has many special properties \cite{GroConn}, in particular 
\begin{align*}
 \mathcal{L}_{\hat{Z}}\theta = 0, \; \mathcal{L}_{\hat{Z}}\eta = 0 \text{.}
\end{align*}
Moreover, by the very definition of $\hat{Z}$ we have $R_{a*}\hat{Z} = \hat{Z}$ for every $a\in O(2n)$. 

Now suppose that $\xi$ is an infinitesimal isometry. Then there is a certain distinguished class of connections with vanishing vertical torsion.
\begin{lemma}\label{corVertTor}
 Suppose that $\xi$ is an infinitesimal isometry and let $\Gamma$ be a connection on $O_{H,g}(M)$ such that $\Sp\{\hat{\xi}\}\subset\Gamma$. Then the vertical torsion of $\Gamma$ vanishes. 
\end{lemma}
\begin{proof}
 It is enough to show that $\nabla_{\xi}X = [\xi,X]$ for every $X\in\Sec(H)$ -- cf. the proof of \cite[Proposition 5.1]{GroConn}.
\end{proof}

Under the assumptions of the lemma $\nabla_ZX = P[Z,X]$ whenever $Z\in\Sp\{\xi\}, X\in\Sec(H)$.

A contact sub-Riemannian manifold $(M,H,g)$ will be called \textit{special} if it is oriented as a contact manifold and the Reeb vector field $\xi$ is an infinitesimal isometry.

Fix a special sub-Riemannian manifold $(M,H,g)$ and take $\Gamma$ to be a connection on $O_{H,g}(M)$ chosen as in Lemma \ref{corVertTor}. By the lemma we know that the vertical torsion of $\Gamma$ vanishes. It turns out (see \cite{GroConn}) that by a suitable modification of $\Gamma^H$ we can get rid of the horizontal torsion as well, and the connection obtained in this way is unique. Our considerations may be summed up as follows.
\begin{theorem}[\cite{GroConn}]\label{SRCon}
On every special sub-Riemannian manifold there exists a unique metric and torsion-free sub-Riemannian connection.
\end{theorem}

Up to the end of this section we suppose that $(M,H,g)$ is a special sub-Riemannian manifold and $\Gamma$ is the connection from Theorem \ref{SRCon}. Denote by $\omega$ the corresponding connection form, i.e., $\Gamma = \ker\omega$, $R_a^*\omega = Ad_{a^{-1}}\cdot\omega$ for all $a\in O(2n)$, and $\omega(A^*) = A$ for every $A\in o(2n)$, where $A^*$ is the fundamental vector field defined by $A$. By the \textit{curvature form} of $\Gamma$ we mean a $2$-form $\Omega$ on $O_{H,g}(M)$ with values in $o(2n)$ defined as
\begin{align*}
 \Omega = d\omega\circ(p_{\Gamma},p_{\Gamma})\text{.}
\end{align*}
$\Omega$ determines the curvature tensor defined in the following way. For $Z,W\in\Sec(TM)$ we consider an operator $R(Z,W)\colon\Sec(H)\longrightarrow\Sec(H)$ given by
\begin{align*}
 (R(Z,W)X)(q) = l(\Omega_l(Z^*,W^*)(l^{-1}X))\text{,}
\end{align*}
where $l\in\pi^{-1}(q)$ and $Z^*$ and $W^*$ are, respectively, lifts of $Z$ and $W$ to $O_{H,g}(M)$. Similarly as in the classical case one proves the formula
\begin{align}\label{Rform}
 R(Z,W)X = \nabla_Z\nabla_WX - \nabla_W\nabla_ZX - \nabla_{[Z,W]}X \text{.}
\end{align}
There also hold the following horizontal versions of Bianchi identities.
\begin{proposition}\label{Bianchi}
 For every $X,Y,Z\in\Sec(H)$
 \begin{align*}
  R(X,Y)Z + R(Y,Z)X + R(Z,X)Y = 0
 \end{align*}
 (the first identity), and
 \begin{align*}
 (\nabla_XR) (Y,Z) +(\nabla_YR) (Z,X) + (\nabla_ZR) (X,Y) = 0
 \end{align*}
 (the second identity).
\end{proposition}
Note that the above formulae cease to make sense when one of the arguments is not horizontal.

\begin{proposition}\label{prop2}
 On every special sub-Riemannian manifold $R(\xi,W) = 0$ for any $W\in\Sec(TM)$.
\end{proposition}
\begin{proof}
 Let $\hat{\xi}$ be the $\Gamma$-horizontal lift of the Reeb vector field $\xi$ as defined above. Moreover let $W^*$ be the $\Gamma$-horizontal lift of $W$. According to \cite{GroConn} we know that $\mathcal{L}_{\hat{\xi}}\omega = 0$. Then 
 \begin{align*}
  0 = (\mathcal{L}_{\hat{\xi}}\omega)(W^*) = \hat{\xi}(\omega(W^*)) - \omega([\hat{\xi},W^*])
 \end{align*}
from which $[\hat{\xi},W^*]$ is again $\Gamma$-horizontal. As a consequence $\Omega(\hat{\xi},W^*) = 0$ which ends the proof.
\end{proof}
By the way we see that $[\hat{\xi},\Sec(\Gamma)]\subset\Sec(\Gamma)$.
We note one more result which we will need below.
\begin{proposition}\label{skewR}
 For every $Z,W\in\Sec(TM)$ the operator $R(Z,W)$ is skew-symmetric with respect to $g$, i.e., $g(R(Z,W)X,Y) + g(X,R(Z,W)Y) = 0$ for all $X,Y\in\Sec(H)$
\end{proposition}
\begin{proof}
 The proof is exactly the same as in the classical case.
\end{proof}

Denote by $\mathfrak{H}$ the algebra of all horizontal mixed tensor fields on $M$. It means that elements of $\mathfrak{H}$ are precisely the sections of the bundles 
\begin{align*}
 \underbrace{H\otimes\dots\otimes H}_{i\; \mathrm{factors}} \otimes \underbrace{H^*\otimes\dots\otimes H^*}_{j\;\mathrm{factors}}\longrightarrow M \text{,}
\end{align*}
$i,j\geq 0$.
It follows from the proposition above that the value of $R$ at a point $q\in M$ is entirely determined by its component in $\bigwedge^2H_q^*\otimes H_q^*\otimes H_q$. Hence we may regard $R$ as an element of $\mathfrak{H}$. Let us assume the usual definition of the covariant differential, namely $\nabla R(Z;X,Y) = (\nabla_ZR)(X,Y)$. Although $\nabla R$ is well defined as a section of the bundle $T^*M\otimes\bigwedge^2H^*\otimes H^*\otimes H\longrightarrow M$, some care is needed when we want to consider higher-order differentials $\nabla^iR$, $i=1,2,\dots$ It turns out that $\nabla^iR$, $i\geq 2$, makes sense only when we apply to it horizontal vector fields (cf. the formula in \cite[Proposition 2.12, p.125]{KoNo}).
Therefore, in order to make sure that everything works, in the sequel \textbf{we will consider $\boldsymbol{\nabla^iR$, $i=0,1,\dots}$, as elements of $\mathfrak{H}$} (i.e., ignoring non-horizontal components). The same remark may be applied to $\nabla^id\alpha$, $i=0,1,\dots$ Then, for every positive integer $i$ and $Z\in\Sec(TM)$, expressions like $\nabla_Z(\nabla^iR)$, $\nabla_Z(\nabla^id\alpha)$ or $\mathcal{L}_Z(\nabla^iR)$, $\mathcal{L}_Z(\nabla^id\alpha)$, with $Z$ being a contact vector field in the two latter cases, are well defined elements of $\mathfrak{H}$.

\begin{proposition}\label{LieR}
 For every infinitesimal isometry $Z$, $\mathcal{L}_Z(\nabla^iR) = 0$ and $\mathcal{L}_Z(\nabla^id\alpha) = 0$, $i=0,2,\dots$
\end{proposition}
\begin{proof}
Let $\varphi^t$ be the (local) flow of $Z$. The result follows from (\ref{Rform}) and from the relation $\varphi^t_*(\nabla_XY) = \nabla_{\varphi^t_*X}\varphi^t_*Y$. This last formula can be obtained by direct calculations using the definition of $\nabla$. Then we proceed by induction: $\mathcal{L}_Z(\nabla^iR) = 0$ implies $\mathcal{L}_Z(\nabla^{i+1}R) = 0$.

For the second formula we note that $\mathcal{L}_Z\alpha = 0$ (see Proposition \ref{PropIsom}) and again proceed by induction.
\end{proof}

\section{Prolongation of infinitesimal isometries}\label{SecProl}
In this section we adapt some notions and definitions from the classical differential geometry to the setting of special sub-Riemannian manifolds and investigate their properties. Then, using the obtained results and  making use of the sub-Riemannian connection, we generalize some results obtained in \cite{Nom} to the sub-Riemannian situation.
\subsection{Operator $A_Z$}\label{subSecAZ}
Fix a special sub-Riemannian manifold $(M,H,g)$, $\dim M = 2n+1$. Recall that $P\colon TM\longrightarrow H$ is the projection determined the splitting $TM = H\oplus \Sp\{\xi\}$. Let $\Gamma$ be the sub-Riemannian connection from Theorem \ref{SRCon} and denote by $\nabla$ the corresponding covariant differentiation.  For a contact vector field $Z$ we define the operator $\widetilde{A}_Z\colon \Sec(TM)\longrightarrow\Sec(H)$, linear over $C^{\infty}(M)$, by formula
\begin{align*}
 \widetilde{A}_Z = P\circ\mathcal{L}_Z - \nabla_Z\circ P\text{.}
\end{align*}
$\widetilde{A}_Z$ can also be regarded as a bundle morphism $\widetilde{A}_Z:TM\longrightarrow H$ covering the identity. The restriction of $\widetilde{A}_Z$ to $\Sec(H)$ (respectively, to $H$) will be denoted by $A_Z$. So we have
\begin{align} \label{OPAZ}
 A_Z = \mathcal{L}_Z - \nabla_Z\colon \Sec(H)\longrightarrow\Sec(H)
\end{align}
or, respectively, $A_Z\colon H\longrightarrow H$ as a bundle morphism. Remark that if $A$ is an infinitesimal isometry then $\widetilde{A}_Z$ is $0$ on $\Sp\{\xi\}$. Thus on every special sub-Riemannian manifold the operator $\widetilde{A}_Z$, where $Z$ is an infinitesimal isometry, can be viewed as the trivial (i.e. zero) extension of $A_Z$, where $A_Z$ is given by (\ref{OPAZ}). Therefore, in the sequel, we will not use the tilde symbol, agreeing that $A_Z\xi = 0$. The reason why we consider the value of $A_Z$ on $\Sp\{\xi\}$ is motivated by the statement of Proposition \ref{propEqs} below and its applications, for instant Theorem \ref{th1} whose proof significantly simplifies if we are allowed to use all curves, not only horizontal ones.

For a point $q\in M$ we will write $A_Z(q):H_q\longrightarrow H_q$. The value assumed on a vector $v\in H_q$ will be written as $A_Z(q)v$.

The operator $A_Z$ has many important properties. First of all, it is easy to verify the following proposition (c.f. Proposition \ref{InfIsmCon} and formula (\ref{MetrCon})).
\begin{proposition}\label{prop1}
 Suppose that $Z\in\Sec(TM)$ is a contact vector field. Then $Z$ is an infinitesimal isometry if and only if $A_Z$ is a skew-symmetric operator with respect to $g$, i.e., $g(A_ZX,Y) + g(X,A_ZY) = 0$ for all $X,Y\in\Sec(H)$. 
\end{proposition}
\begin{corollary}\label{cor1}
 Let $Z$ be a contact vector field. If $\nabla_{\xi}A_Z = 0$, then $[Z,\xi]$ is an infinitesimal isometry.
\end{corollary}
\begin{proof}
 For every $X\in\Sec(H)$ we have
 \begin{gather*}
  (\nabla_{\xi}A_Z)X = \nabla_{\xi}(A_ZX) - A_Z([\xi,X]) = [\xi,[Z,X]] - \nabla_{\xi}\nabla_{Z}X - [Z,[\xi,X]] + \nabla_Z\nabla_{\xi}X=\\
  R(Z,\xi)X + \nabla_{[Z,\xi]}X + [\xi,[Z,X]] - [Z,[\xi,X]] = \nabla_{[Z,\xi]}X - [[Z,\xi],X] = 0 \text{,}
 \end{gather*}
 where we used Proposition \ref{prop2} and the Jacobi identity. As a result we see that $\nabla_{[Z,\xi]}X = [[Z,\xi],X]$ for every $X\in\Sec(H)$ which means that $A_{[Z,\xi]} = 0$. Thereby Proposition \ref{prop1} applies.
\end{proof}
Using Proposition \ref{PropIsom} and the computations from the proof of Corollary \ref{cor1} we immediately obtain
\begin{corollary}\label{cor3}
 If $Z$ is an infinitesimal isometry then $\nabla_{\xi}A_Z = 0$.
\end{corollary}

Further we generalize Proposition \ref{prop1}.

\begin{proposition}\label{prop4}
 If $Z$ is an infinitesimal isometry then $\nabla_VA_Z$ is skew-symmetric with respect to $g$ for every $V\in\Sec(TM)$.
\end{proposition}
\begin{proof}
 By assumption the operator $A_Z$ is skew-symmetric. Let us differentiate the equation
 \[
  g(A_ZX,Y) + g(X,A_ZY) = 0\text{,}
 \]
$X,Y\in\Sec(H)$, in the direction $V$. As a result we have
\begin{gather*}
 g(\nabla_V(A_ZX),Y) + g(A_ZX,\nabla_VY) + g(\nabla_VX,A_ZY) + g(X,\nabla_V(A_ZY)) = 0
\end{gather*}
which after computations gives
\begin{gather*}
 0 = g((\nabla_VA_Z)X,Y) + g(X,(\nabla_VA_Z)Y) + \\
 g(A_Z(\nabla_VX),Y) + g(\nabla_VX,A_ZY) + g(A_ZX,\nabla_VY) + g(X,A_Z(\nabla_VY))\text{.}
\end{gather*}
This concludes the proof since by Proposition \ref{prop1} the last four terms add up to zero.
\end{proof}

Before we come to the key result in this section, we will prove the following lemma.
\begin{lemma}\label{lem1}
 Suppose that $Z\in\Sec(TM)$ is an infinitesimal isometry. Then for every $X,Y\in\Sec(H)$
 \[
  [Z,P[X,Y]] = P[Z,[X,Y]]\text{.} 
 \]
\end{lemma}
\begin{proof}
 We have $[X,Y] = P[X,Y] + h\xi$ for a smooth function $h$. Remembering that $\xi$ commutes with infinitesimal isometries we have
 \[
  [Z,[X,Y]] = [Z,P[X,Y]] + Z(h)\xi\text{.}
 \]
 Since $[Z,P[X,Y]]\in \Sec(H)$, we end the proof by applying $P$ to both sides of the above equality. 
\end{proof}

\begin{proposition}\label{prop3}
 Suppose that $Z$ is an infinitesimal isometry. Then 
 \[
 \nabla_VA_Z = R(Z,V) 
 \]
for every $V\in\Sec(TM)$.
\end{proposition}
\begin{proof}
 There are two cases. At first we treat the easier one, i.e., $V = \xi$. By Corollary \ref{cor3} we know that $\nabla_{\xi}A_Z = 0$, while $R(Z,\xi) = 0$ by Proposition \ref{prop2}. 

Now suppose that $V\in\Sec(H)$. First, similarly as in \cite{Kostant}, we will show that the expression $\nabla_Y(A_Z)X - R(Z,Y)X$ is symmetric with respect to $X$ and $Y$, for $X,Y\in\Sec(H)$. To this end let us fix $X,Y\in\Sec(H)$. Then
\begin{align*}
 (\nabla_YA_Z)X = \nabla_Y(A_ZX) - A_Z(\nabla_YX) = \nabla_Y[Z,X] - \nabla_Y\nabla_ZX - A_Z(\nabla_YX)
\end{align*}
and similarly
\begin{align*}
 (\nabla_XA_Z)Y =  \nabla_X[Z,Y] - \nabla_X\nabla_ZY - A_Z(\nabla_XY)\text{.}
\end{align*}
Subtracting the second equation from the first we have
\begin{gather*}
 (\nabla_YA_Z)X - (\nabla_XA_Z)Y = \smallskip\\
 A_Z(P[X,Y]) + \nabla_Y[Z,X] - \nabla_X[Z,Y] + \nabla_X\nabla_ZY - \nabla_Y\nabla_ZX\text{.}
\end{gather*}
Using (\ref{Rform}) we are led to
\begin{gather*}
 (\nabla_YA_Z)X - (\nabla_XA_Z)Y = \smallskip\\
 A_Z(P[X,Y]) + P[Y,[Z,X]] + P[[Z,Y],X] + R(X,Z)Y + R(Z,Y)X + \nabla_ZP[X,Y] = \smallskip\\
 [Z,P[X,Y]] + P[Y,[Z,X]] + P[[Z,Y],X] + R(X,Z)Y + R(Z,Y)X\text{.}
\end{gather*}
Using Lemma \ref{lem1} and the Jacobi identity, the first three terms in the last equation sum up to zero, therefore we obtain
\begin{align*}
 (\nabla_YA_Z)X - (\nabla_XA_Z)Y = R(X,Z)Y + R(Z,Y)X
\end{align*}
which finally gives
\begin{align*}
 (\nabla_YA_Z)X - R(Z,Y)X =  (\nabla_XA_Z)Y - R(Z,X)Y\text{.}
\end{align*}
Now let us consider the following $3$-linear form
\begin{align}
 \Psi(Y,X,W) = g\big((\nabla_YA_Z)X - R(Z,Y)X, W\big)\text{.}
\end{align}
As we already know $\Psi$ is symmetric with respect to $X$ and $Y$. Moreover
\begin{gather*}
 \Psi(Y,X,W) + \Psi(Y,W,X) = \\
 g((\nabla_YA_Z)X,W) - g(R(Z,Y)X, W) + g((\nabla_YA_Z)W,X) - g(R(Z,Y)W, X)
\end{gather*}
where the latter expression equals zero because $R(Z,Y)$ and $\nabla_Y(A_Z)$ are skew-symmetric by Propositions \ref{skewR}, \ref{prop4}. Therefore $\Psi$ is skew-symmetric with respect to $X$ and $W$. By the so-called \textit{S3 lemma} $\Psi$ vanishes identically which ends the proof.
\end{proof}

Let $Z$ be an infinitesimal isometry. In the sequel we will make use of the following decomposition 
\begin{align}\label{DecompNot}
 Z = PZ + f_Z\xi\text{.}
\end{align}
where $P$ is the projection (\ref{project}) and $f_Z$ is an appropriate smooth function. Note that if $Z_1,Z_2$ are infinitesimal isometries such that their horizontal parts are equal, i.e., $PZ_1 = PZ_2$, and moreover $Z_1(q_0) = Z_2(q_0)$ at a point $q_0\in M$, then $Z_1 = Z_2$. Indeed, we have $Z_i = Y + f_{Z_i}\xi$, $i=1,2$, $Y\in\Sec(H)$. Now $Z_1 - Z_1 = (f_{Z_2} - f_{Z_1})\xi$ is an isometry which implies (cf. Proposition \ref{PropIsom}) that $f_{Z_2} = f_{Z_1} + c$ for a constant $c$. However $f_{Z_1}(q_0) = f_{Z_2}(q_0)$, so $c = 0$. Let also note that if $Z = PZ + f_Z\xi$ is an infinitesimal isometry then
\begin{align}\label{r2}
 [\xi,PZ] = 0,\;\; \xi(f_Z) = 0\text{.}
\end{align}
Actually, $0 = [Z,\xi] = [PZ,\xi] -\xi(f_Z)\xi$, where the first summand, when evaluated at a point, is in $H$ while the other is in $\Sp\{\xi\}$. 

\begin{lemma}\label{lem2}
 Let $Z$ be an infinitesimal isometry. Then for every $V\in\Sec(TM)$ 
 \begin{align*}
 \nabla_VPZ = -A_ZV\text{.} 
 \end{align*}
\end{lemma}
Note that although $V$ does not have to be horizontal, the above formula still makes sense. It follows from (\ref{r2}) as well as from the remark at the beginning of subsection \ref{subSecAZ} that both sides are zero on $\Sp\{\xi\}$.
\begin{proof}[Proof of Lemma \ref{lem2}.]
 Indeed, for $V\in\Sec(H)$ the formula follows from the sequence of equalities
 \begin{gather*}
 \nabla_VPZ = \nabla_{PZ}V + P([V,PZ]) = \nabla_{PZ}V + P([V,Z - f_Z\xi]) =  \\ 
 \nabla_{Z - f_Z\xi}V + [V,Z] - f_Z[V,\xi] = \nabla_ZV - f_Z\nabla_{\xi}V - [V,Y] + f_Z[\xi,V] = -A_ZV \text{.}
\end{gather*}
\end{proof}

Recall that $\alpha$ stands for the normalized contact form defined by the sub-Riemannian structure that we consider. From the very definition of the Reeb vector field we have $d\alpha(V,W) = d\alpha(PV,W)$ for all $V,W\in\Sec(TM)$. 
As a result of the above considerations we are ready to prove a proposition describing an important for the future use observation.
\begin{proposition}\label{propEqs}
Let $\gamma\colon [a,b]\longrightarrow M$ be a smooth curve and let $Z$ be an infinitesimal isometry. Then the following system of ordinary differential equations is satisfied
 \begin{align}\label{r1}
 \begin{array}{l}
  \nabla_{\dot{\gamma}(t)}PZ = -A_Z(\gamma(t))\dot{\gamma}(t) \smallskip\\
  \nabla_{\dot{\gamma}(t)}A_Z = R(PZ(\gamma(t)), \dot{\gamma}(t)) \smallskip\\
  \nabla_{\dot{\gamma}(t)}f_Z = -d\alpha(PZ(\gamma(t)), \dot{\gamma}(t))
 \end{array}
\end{align}
along $\gamma$.
\end{proposition}
\begin{proof}
In order to obtain the second equation in (\ref{r1}) we use Proposition \ref{prop3} and Proposition \ref{prop2}:
\begin{align*}
 R(Z(\gamma(t)), \dot{\gamma}(t)) = R(PZ(\gamma(t)), \dot{\gamma}(t))\text{.}
\end{align*}
On the other hand, the last equation readily follow from $\mathcal{L}_Z\alpha = 0$. Indeed, by Cartan's formula
 \begin{gather*}
  df_Z + i_Zd\alpha = 0
 \end{gather*}
and it is enough to evaluate the above equation for $\dot{\gamma}(t)$ and use the remark before the statement of Proposition \ref{propEqs}.
\end{proof}
From the third equation in (\ref{r1}) we see that, in accordance with what was said above, $f_Z$ is uniquely determined by $PZ$ and by the value $Z(q_0)$ for an arbitrarily chosen point $q_0$.
\begin{corollary} 
 Suppose that $Z\in \iso(U)$, where $U$ is a connected open set. Then $Z$ is uniquely determined by the value of $Z$ and $A_Z$ at a single point of $U$.
\end{corollary}
\begin{proof}
 Take $q_0\in U$ and an arbitrary smooth curve $\gamma:[a,b]\longrightarrow U$, $\gamma(a) = q_0$. Given $PZ(q_0)$, $f_Z(q_0)$, and $A_Z(q_0)$, the system (\ref{r1}) uniquely determines the values of $Z$ along $\gamma$.
\end{proof}

\subsection{Prolongation}\label{subscProl}
For a point $q\in M$ denote by $\mathfrak{a}(q)$ the following vector space:
\begin{align*}
 \mathfrak{a}(q) = H_q \oplus E_q \oplus \R \text{,}
\end{align*}
where $E_q$ stands for the set of all linear operators $H_q\longrightarrow H_q$ which are skew-symmetric with respect to $g$. If $U$ is a connected open set and $q\in U$ then the mapping
\begin{align}\label{mapp}
 \iso(U) \ni Z \longrightarrow (PZ(q), A_Z(q), f_Z(q)) \in \mathfrak{a}(q)
\end{align}
is injective for every $p\in U$ by the above considerations.

Now we introduce the notion of prolongation of infinitesimal isometries.

\begin{definition}
 Let $\gamma:[0,1]\longrightarrow M$ be a smooth curve. Take $Z^*\in\is^*(\gamma(0))$. By a prolongation of $Z^*$ along $\gamma$ we mean a family of germs $Z^*(t)\in\is^*(\gamma(t))$ such that $Z^*(0)=Z^*$ and moreover the following condition is satisfied. For every $t$, if $Z_t$ is a representative of $Z^*(t)$ defined on a neighborhood of $\gamma(t)$, it determines the element $(X(t),A(t),c(t))\in\mathfrak{a}(\gamma(t))$$:$ $X(t)=PZ_t(\gamma(t))$, $A(t) = A_{Z_t}(\gamma(t))$, $c(t) = f_{Z_t}(\gamma(t))$. We then require that the system of ordinary differential equations
 \begin{align}\label{EqDefProl}
 \begin{array}{l}
  \nabla_{\dot{\gamma}(t)}X(t) = -A(t)\dot{\gamma}(t) \smallskip\\
  \nabla_{\dot{\gamma}(t)}A(t) = R(X(t)), \dot{\gamma}(t)) \smallskip\\
  \nabla_{\dot{\gamma}(t)}c(t) = -d\alpha(X(t),\dot{\gamma}(t))
 \end{array}
\end{align}
is satisfied along $\gamma$.
\end{definition}

\begin{proposition}
Let $\gamma:[0,1]\longrightarrow M$ be a smooth curve and take $Z^*\in\is^*(\gamma(0))$. If a prolongation of $Z^*$ along $\gamma$ exists, then it is unique.
\end{proposition}
\begin{proof}
 Two prolongations satisfy the same system of ordinary differential equations with the same initial conditions.
\end{proof}
\begin{proposition}
 Suppose that $\gamma:[0,1]\longrightarrow M$ is a smooth curve which contains solely $\is^*$-regular points. Then for every $Z^*\in\is^*(\gamma(0))$ the prolongation along $\gamma$ exists.
\end{proposition}
\begin{proof}
 The result is crucial for the paper, so we present its proof, although the reasoning is quite standard and can be found, with obvious modifications, in \cite{Nom}.
 
 Let $\gamma:[0,1]\longrightarrow M$ be a curve as in the hypothesis of the proposition and take $Z^*\in\is^*(\gamma(0))$. Let $Z$ be a representative of $Z^*$ defined on a neighborhood of $\gamma(0)$. Clearly $Z$ defines the elements 
\begin{align}\label{fam1}
 (X(t),A(t),c(t)) = (PZ({\gamma(t)}), A_Z({\gamma(t)}), f_Z(\gamma(t))) \in \mathfrak{a}(\gamma(t)) 
\end{align}
on some interval $[0,t_1)$ and the equations (\ref{EqDefProl}) are satisfied. It means that the prolongation $Z^*(t)$ of $Z^*$ exits on a certain interval $[0,t_1)$, $t_1>0$. Let
 \begin{align*}
  t_0 = \sup\{t_1>0:\; \textit{the prolongation exists on $[0,t_1)$}\}\text{.}
 \end{align*}
Obviously, we can assume that (\ref{fam1}) is defined on $[0,t_0)$. Denote by $U$ an $\iso$-special neighborhood of $\gamma(t_0)$ such that $U\ni p\longrightarrow\dim\is^*(p)$ is constant. We find an $\varepsilon>0$ such that $\gamma([t_0-\varepsilon, t_0+\varepsilon])\subset U$. Pick a number $t_1$, $t_0-\varepsilon<t_1<t_0$, and denote by $Z_1$ an element of $\iso(U)$ such that $(Z_1)_{\gamma(t_1)} = Z^*(t_1)$ (cf. Proposition \ref{prop01}). The family 
\begin{align}\label{fam2}
 (X_1(t), A_1(t), c_1(t)) \in \mathfrak{a}(\gamma(t))\text{,}
\end{align}
$t_0 - \varepsilon < t < t_0 + \varepsilon$, induced by $Z_1$ satisfies the equations (\ref{EqDefProl}). Now, in order to construct the prolongation of $Z^*$ beyond $t_0$, it is enough to show that the two families (\ref{fam1}) and (\ref{fam2}) coincide on $[t_0 - \varepsilon,t_0)$. This is, however, clear because they both satisfy the same system of differential equations (\ref{EqDefProl}) and, moreover, $(X_1(t_1), A_1(t_1), c_1(t_1)) = (X(t_1),A(t_1),c(t_1))$ by the choice of the field $Z_1$. 
\end{proof}

\subsection{Proof of Theorem \ref{th1}}
 Everything will become clear if we show that the prolongation procedure along a curve does not depend on a curve itself but only on its endpoints. Indeed, it is true in simply connected $\iso$-special neighborhoods, since then, by Proposition \ref{prop01}, all germs that arise during the prolongation are germs determined by a single vector field. Take two points $q_0,q_1\in M$, and let $\gamma_0, \gamma_1\colon [0,1]\longrightarrow M$ be such smooth curves that $\gamma_0(0) = \gamma_1(0) = q_0$, $\gamma_0(1) = \gamma_1(1) = q_1$. By assumption, $\gamma_0$ and $\gamma_1$ can be embedded into a family of curves $\gamma_s\colon[0,1]\longrightarrow M$, $s\in[0,1]$, such that $\gamma_s(0) = q_0$, $\gamma_s(1)=q_1$ and the mapping $(t,s)\longrightarrow\gamma_s(t)$ is continuous. Suppose that $\{U_i\}_{i=1,\dots,m}$ is a finite cover of $\gamma_0([0,1])$ by simply connected $\iso$-special neighborhoods. Then we find a number $\delta>0$ such that $\gamma_s([0,1])\subset\bigcup_{i=1}^m U_i$ for $s<\delta$. Using appropriate deformations we see that the prolongations from $q_0$ to $q_1$ along $\gamma_0$ and $\gamma_s$ coincide for $s<\delta$. The standard argument shows now that this is true for every $s\in[0,1]$.
 
Now, starting from an arbitrary point $q_0\in M$ and a germ $Z^*(q_0)\in\is^*(q_0)$ we build a field of germs $Z^*(q)\in\is^*(q)$ on $M$, where each $Z^*(q)$ is the prolongation of $Z^*(q_0)$ along a smooth curve going from $q_0$ to $q$. For each $q\in M$, similarly as in (\ref{mapp}), we have an injection $\is^*(q)\ni Z^*(q)\longrightarrow (PZ(q),A(q),c(q))\in\mathfrak{a}(q)$ which permits to construct a globally defined vector field $Z(q) = PZ(q) + c(q)\xi(q)$. We argue that $Z$ is an infinitesimal isometry. Take a point $q\in M$. By Proposition \ref{prop01} we find a neighborhood $U$ of $q$ and $W\in\iso(U)$ such that $Z^*(q) = (W)_q$. Using what we have said above $W$ and $Z$ coincide on $U$ proving that $Z$ is an infinitesimal isometry around $q$. This ends the proof.

\section{$\is$-regular points}\label{SecAnalytic}

In this section we prove Theorems \ref{th2} and \ref{th3}, again following \cite{Nom}. As above our main tool is the sub-Riemannian connection introduced in Section \ref{sRCON}. We keep the same notation, i.e., $(M,H,g)$ is a fixed special sub-Riemannian manifold  which is smooth unless otherwise stated, and $\nabla$, $\xi$, $\alpha$ etc. are as in the previous sections.


\subsection{$\iso$-generators}
Let $\mathfrak{H}$ be the algebra of all horizontal mixed tensor fields on $M$ (cf. the end of Section \ref{sRCON}). Also, for a point $q\in M$, let $\mathfrak{H}_q$ stand for the mixed tensor algebra over $H_q$. Similarly as it is done in \cite{Kostant} we adopt the following definition. Fix a point $q\in M$. By an \textit{$\mathfrak{H}$-derivation at $q$} we mean any $\R$-linear mapping $D\colon\mathfrak{H}\longrightarrow\mathfrak{H}_q$ which satisfies
\begin{enumerate}
 \item \label{P1} $D(T_1\otimes T_2) = T_1(q)\otimes (DT_2) + (DT_1)\otimes T_2(q)$ for every $T_1,T_2\in\mathfrak{H}$;
 \item \label{P2} $D$ preserves the type of tensors;
 \item \label{P3} $D$ commutes with any contraction.
\end{enumerate}
For instance, if $v\in T_qM$ then $\nabla_v$ induces such a derivation. Also $\mathcal{L}_Z$ induces this type of derivation, where $Z$ is an infinitesimal isometry. And one more example. Let $A\colon H_q\longrightarrow H_q$ be an endomorphism. Then $A$ defines an $\mathfrak{H}$-derivation at $q$ by simply setting $Af = 0$ for $f\in C^{\infty}(M)$ and $AX = A(X(q))$ for $X\in\Sec(H)$. Then, for instance, by (\ref{P1}) and (\ref{P3}) we have $A\omega = -A^*(\omega(q))$ for $\omega\in\Sec(H^*)$, where $A^*$ is the transpose of $A$. It is proved that every $\mathfrak{H}$-derivation at $q$ which is zero on functions is generated by an endomorphism $H_q\longrightarrow H_q$.

A triple $(X,A,c) \in\mathfrak{a}(q)$ (see subsection \ref{subscProl}) is called an \textit{$\iso$-generator at $q$}, if the derivation at $q$ induced by $\nabla_{X+c\xi(q)} + A$ satisfies
\begin{gather*}
 (\nabla_{X+c\xi(q)} + A)(\nabla^iR) = 0,\; (\nabla_{X+c\xi(q)} + A)(\nabla^id\alpha) = 0
\end{gather*}
for every $i=0,1\dots$.
(here $\nabla^0$ is just the identity). Recall that in accordance with the agreement that we made at the end of Section \ref{sRCON} we view each $\nabla^iR$ as an element of $\mathfrak{H}$. Exactly the same remark applies to $\nabla^id\alpha$. The set of $\iso$-generators at $q$ will be denoted by $\is(q)$. It is clear from Proposition \ref{LieR} that if $Z$ is an infinitesimal isometry then $(PZ(q),A_Z(q),f_Z(q))\in\is(q)$, that is to say we have a linear injection $\is^*(q)\longrightarrow\is(q)$.

In order to investigate further properties of $\is(q)$ we proceed similarly as in \cite{Nom}. For an integer $m$ we introduce the spaces 
\begin{gather*}
 \is_m(q)=\{(X,A,c)\in\mathfrak{a}(q):\;\\
 (\nabla_{X+c\xi(q)}+A)(\nabla^iR) = 0, (\nabla_{X+c\xi(q)}+A)(\nabla^id\alpha), i=0,1,\dots m\} \text{.}
\end{gather*}
By the very definition $\is_{m+1}(q)\subset\is_m(q)\subset\mathfrak{a}(q)$ for every $m$, and $\is(q) = \bigcap_{m=1}^{\infty}\is_m(q)$. Thus the sequence $\{\dim\is_m(q)\}$ is non-increasing, so it stabilizes from a certain moment, i.e., there exists an integer $m_0$ (depending on $q$) such that $\dim\is(q) = \dim\is_m(q)$ for all $m>m_0$.
\begin{lemma}\label{lem3}
 Take a point $q_0\in M$. There exits an integer $m$ and a neighborhood $U$ of $q_0$ with the property that 
 \[
  \dim\is(q)\leq\dim\is_m(q)\leq\dim\is_m(q_0)=\dim\is(q_0)
 \]
 for every $q\in U$.
\end{lemma}
\begin{proof}
 Take $m$ to be an arbitrary integer. For integers $r,s\geq 0$ denote by $\mathfrak{H}^{r,s}_q$ the set of elements of $\mathfrak{H}_q$ of type $(r,s)$. Let moreover
 \[
  \mathfrak{H}^i_q = \mathfrak{H}^{1,3+i}_q \oplus \mathfrak{H}^{0,2+i}_q
 \]
and define a mapping $f_i\colon\mathfrak{a}(q)\longrightarrow \mathfrak{H}^i_q$ by formula
 \[
 f_i(X,A,c) = (\nabla_{X+c\xi(q)}+A)(\nabla^iR) + (\nabla_{X+c\xi(q)}+A)(\nabla^id\alpha)\text{.} 
 \]
Setting $S(q) = \mathfrak{H}^0_q\oplus \mathfrak{H}^1_q\oplus\dots\oplus \mathfrak{H}^m_q$, consider a linear mapping $f_q\colon\mathfrak{a}(q)\longrightarrow S(q)$ given by $f_q(X,A,c) = f_0(X,A,c)+\dots+f_m(X,A,c)$. We see that $\ker f_q = \is_m(q)$. After these preparatory steps the proof is standard and goes without changes as in \cite{Nom}.
\end{proof}
A point $q\in M$ will be called \textit{$\is$-regular} if the function $p\longrightarrow\dim\is(p)$ is constant in a neighborhood of $q$. The set of $\is$-regular points will be denoted by $M^o$.
\begin{proposition}\label{OpenDens}
 The set $M^o$ is open and dense in $M$.
\end{proposition}
\begin{proof}
 $M^o$ is open by the very definition. Take a point $q_0\in M$ and let $U$ be its arbitrarily small neighborhood as in Lemma \ref{lem3}. Then $\dim\is(p)\leq \dim\is(q_0)$ for every $p\in M$. Now, let $q\in U$ be such that 
 \[
  \dim\is(q) = \min\{\dim\is(p):\;p\in U\} \text{,}
 \]
and let $V\subset U$ be a neighborhood of $q$ as in Lemma \ref{lem3}. By the choice of $q$ we have $\dim\is(q) = \dim\is(p)$ for every $p\in V$ showing that $q$ is $\is$-regular. Hence $M^o\cap U\ne\emptyset$ which ends the proof.
\end{proof}

Now take a point $q\in M$ and a smooth horizontal curve $\gamma\colon[0,1]\longrightarrow M$, $\gamma(0) = q$. For a given $(X,A,c)\in\is(q)$ let $(X(t),A(t),c(t))$ be the solution to the system
\begin{align}\label{rowXAc}
 \begin{array}{l}
  \nabla_{\dot{\gamma}(t)}X(t) = -A(t)\dot{\gamma}(t) \smallskip\\
  \nabla_{\dot{\gamma}(t)}A(t) = R(X(t), \dot{\gamma}(t)) \smallskip\\
  \nabla_{\dot{\gamma}(t)}c(t) = -d\alpha(X(t), \dot{\gamma}(t))
 \end{array}
\end{align}
with initial condition $(X,A,c)$ at $t=0$. 
\begin{lemma}
 Using the above notation $(X(t),A(t),c(t))\in\mathfrak{a}(\gamma(t))$ for every $t$.
\end{lemma}
\begin{proof}
 Again we need some preparatory remarks and then the proof will follow by standard argument \cite{Nom}.  
 
 Therefore it suffices to prove that for each fixed $t_0$ the endomorphism 
 \begin{align*}
 A(t_0)\colon H_{\gamma(t_0)}\longrightarrow H_{\gamma(t_0)} 
 \end{align*}
satisfies
 \begin{align}\label{r3}
  g(A(t_0)V,W) + g(V,A(t_0)W) = 0
 \end{align}
for all $V,W\in H_{\gamma(t_0)}$. To this end take two vectors $V,W\in H_{\gamma(t_0)}$ and consider two horizontal vector fields $V(t),W(t)$ along $\gamma$ such that $V(t)$ (resp. $W(t)$) is the solution to initial value problem $\nabla_{\dot{\gamma}(t)}V(t) = 0$, $V(t_0)=V$ (resp. $\nabla_{\dot{\gamma}(t)}W(t) = 0$, $W(t_0)=W$). Now, if we set $f(t) = g(A(t)V(t),W(t)) + g(V(t),A(t)W(t))$, then by standard calculations, using the properties of $V(t),W(t)$ and equations (\ref{rowXAc}), we obtain
\begin{gather*}
 \nabla_{\dot{\gamma}(t)}f(t) = g(R(X(t),\dot{\gamma}(t))V(t),W(t)) + g(V(t),R(X(t),\dot{\gamma}(t))W(t))
\end{gather*}
which is zero by Proposition \ref{skewR}. In particular $f(t_0) = f(0) = 0$ which gives (\ref{r3}).
\end{proof}

In order to determine whether or not $(X(t),A(t),c(t))\in\is(\gamma(t))$, let us consider the family of $\mathfrak{H}$-derivations $D(t)$ at $\gamma(t)$, $t\in[0,1]$, where $D(t)$ is induced by 
\begin{align*}
 \nabla_{X(t)+c(t)\xi(\gamma(t))} + A(t)\text{.}
\end{align*}
We will apply the derivation $D(t)$ only to tensor fields $S\in\mathfrak{H}$ of the form $S = \nabla^iR$ or $S = \nabla^id\alpha$, $i=0,1,\dots$. Clearly $t\longrightarrow D(t)S$ is a tensor field along $\gamma$. Below, for simplicity, we drop the dependence on $t$ in the notation.

\begin{lemma}\label{lem5}
 Let $U$ be a coordinate neighborhood, $q\in U$. Suppose that $V\in\Sec(U,TM)$ is a non-vanishing vector field and $Y_0\in H_q$. Then there exists $Y\in\Sec(U,H)$ such that (i) $Y(q) = Y_0$ and (ii) $P[V,Y] = 0$.
\end{lemma}
\begin{proof}
 Let $e_1,\dots,2_{2n}$ be a basis of $H$ defined on $U$. Write $[e_i,e_j] = C^k_{ij}e_k + C^0_{ij}\xi$, $[\xi,e_j] = C^k_{0j}e_k$. Let, moreover, $V = V^ie_i + V^0\xi$, $Y_0 = Y^j_0e_j(q)$. We seek $Y$ in the form $Y=Y^je_j$. Then
 \begin{gather*}
  [V,Y] = \big(V(Y^k) + (V^iC^k_{ij} - e_j(V^k) + V^0C^k_{0j})Y^j \big)e_k + \\
  Y^j(V^iC^0_{ij} - e_j(V^0))\xi = I^ke_k + I^0\xi \text{.}
 \end{gather*}
 Now, if we choose coordinates $x^1,\dots,x^{2n+1}$ so that (locally) $V = \poch[x_1]$, the system $I^1 = 0,\dots I^{2n} = 0$ becomes a system of $2n$ linear ordinary differential equations with $2n$ unknown functions $Y^1,\dots,Y^{2n}$ depending on $2n$ parameters, so it has solutions satisfying $Y^j(q) = Y^0_j$. 
  
 On the other hand, $I^0 = 0$ may or may not hold at $q$, so in general $[V,Y]\ne 0$.
\end{proof}

\begin{lemma}\label{lem4}
 Under the above notation
 \begin{align*}
 \nabla_{\dot{\gamma}}(DS) = (D(\nabla S))(\dot{\gamma};)
 \end{align*}
 along $\gamma$.
\end{lemma}
The right-hand side in the conclusion of Lemma \ref{lem4} is to be understood in the following way. Let $\mathfrak{H}^{r,s}$ stand for the set of horizontal tensor fields of type $(r,s)$ and take $S\in\mathfrak{H}^{r,s}$. Then $S$ can be viewed as a multilinear mapping $S\colon\Sec(H)\times\dots\times\Sec(H)\longrightarrow\mathfrak{H}^{r,0}$ ($s$ factors). In turn $\nabla S\colon\Sec(H)\times\dots\times\Sec(H)\times\Sec(H)\longrightarrow\mathfrak{H}^{r,0}$ ($s+1$ factors) and $(\nabla S)(V;)$ simply stands for the mapping $\nabla_VS\colon\Sec(H)\times\dots\times\Sec(H)\longrightarrow\mathfrak{H}^{r,0}$ ($s$ factors).
\begin{proof}[Proof of Lemma \ref{lem4}.]
 The argument is essentially the same as in the classical case with one exception. 
 
 Fix $t$. Using Lemma \ref{lem5} let us extend the vector $\dot{\gamma}(t)$ to a vector field $V\in\Sec(U,H)$, $U$ is a neighborhood of $\gamma(t)$, such that $[\xi,V] = P[\xi,V] = 0$. We also extend $X$, $A$ and $c$, respectively, to elements of $\Sec(U,H)$, $\Sec(U,H^*\otimes H)$, $C^{\infty}(U)$. The proof that $(D(\nabla(S))(V;) = D(\nabla_VS) - \nabla_{DV}S$ goes without changes as compared to the classical case. We will show $\nabla_V(DS) = D(\nabla_VS) - \nabla_{DV}S$. So we have
 \begin{gather*}
 \nabla_V(DS) =  \nabla_V(\nabla_XS) + \nabla_V(AS) = \nabla_V(\nabla_XS) + (\nabla_VA)S + A(\nabla_VS) = \\
 \nabla_V(\nabla_XS) + R(X,V)S + A(\nabla_VS) =\\
 \nabla_X(\nabla_VS) - \nabla_{[X,V]}S + A(\nabla_VS) = \\
 D(\nabla_VS) - \nabla_{[X,V]}S = D(\nabla_VS) - \nabla_{P[X,V]}S \text{,}
 \end{gather*}
 where we used Proposition \ref{LieR} ($\nabla_{\xi} = \mathcal{L}_{\xi}$) and equations (\ref{rowXAc}). Now it is enough to notice, again using equations (\ref{rowXAc}), that
 \begin{gather*}
 P[X,V] = \nabla_XV - \nabla_VX =  \nabla_{X+c\xi}V + AV - \nabla_{c\xi}V = DV - c[\xi,V] = DV\text{.}  
 \end{gather*}
 Thus the formula from the statement of the lemma is proved at $\gamma(t)$. 
\end{proof}

Again let $S=\nabla^iR$ or $S=\nabla^id\alpha$. Denote by $F(t) = D(t)(S)$ and write $V(t) = \dot{\gamma}(t)$. Then, by the above discussion, we have (again we omit dependence on $t$):
\begin{gather*}
 \dot{F} = \nabla_V(DS) = (D(\nabla S))(V;)\text{,}
\end{gather*}
\begin{gather*}
 \ddot{F} = \nabla_V\big((D(\nabla S))(V;)\big) = (\nabla_VD(\nabla S))(V;) + (D(\nabla S))(\nabla_VV;) = \\
 (D(\nabla^2 S))(V;V;) + (D(\nabla S))(\nabla_VV;)
\end{gather*}
and so on.
In general we express $F^{(m)}$ as a combination of $D(\nabla^iS)(Y_1;\dots Y_i;)$ for suitable vector fields $Y_1,\dots,Y_i\in\Sec(H)$, $i\leq m$. 
\begin{proposition}
 Under the above notation. If $(M,H,g)$ is an analytic special sub-Riemannian manifold, $\gamma\colon[0,1]\longrightarrow M$ is an analytic curve, and $(X(t),A(t),c(t))$ is defined as the solution to the system (\ref{rowXAc}) with $(X(0),A(0),c(0))\in\is(\gamma(0))$, then $(X(t),A(t),c(t))\in\is(\gamma(t))$ for every $t$.
\end{proposition}
\begin{proof}
 Fix an arbitrary integer $i$. The curve $F_1(t) = D(t)(\nabla^iR)$ is analytic, and by above computations $F_1^{(m)}(0) = 0$ for every $m=0,1,\dots$. Thus $F_1(t) = 0$ on $[0,1]$. Exactly the same argument applies to the curve $F_2(t) = D(t)(\nabla^id\alpha)$.
\end{proof}

 \subsection{Proof of Theorem \ref{th3}}
 Now we a ready to prove Theorem \ref{th3}. Fix an arbitrary point $q_0\in M$ and let $U$ be a coordinate neighborhood of $q_0$. For a point $q\in U$ we construct a mapping $\is(q_0)\longrightarrow\is(q)$ in the following way. Fix an analytic horizontal curve $\gamma\colon [0,1]\longrightarrow U$ such that $\gamma(0)=q_0$, $\gamma(1)=q$. Starting from an arbitrary element $(X,A,c)\in\is(q_0)$ we construct the family $(X(t),A(t),c(t))\in\is(\gamma(t))$ as the solution to the system (\ref{rowXAc}) with initial condition $(X(0),A(0),c(0)) = (X,A,c)$. We define the mentioned mapping as the assignment  
 \begin{align*}
 \is(q_0)\ni(X,A,c)\longrightarrow (X(1),A(1),c(1))\in\is(q) \text{.} 
 \end{align*}
Such a map is clearly a linear isomorphism which shows that $q\longrightarrow\dim\is(q)$ is constant on $U$, that is to say $q_0$ is $\is$-regular.

\bigskip

\subsection{Proof of Theorem \ref{th2}}
We will divide the proof into several steps. We start with the following lemma.
\begin{lemma}
 Suppose that $(M,H,g)$ is smooth. Let $q_0$ be an $\is$-regular point. There exists a neighborhood $U$ of $q_0$ with the following property. Take a smooth horizontal curve $\gamma\colon[0,1]\longrightarrow U$, $\gamma(0)=q_0$, and starting from an element $(X,A,c)\in\is(q_0)$ construct the family $(X(t),A(t),c(t))$ along $\gamma$ by use of equations (\ref{rowXAc}). Then $(X(t),A(t),c(t))\in\is(\gamma(t))$ for every $t$.
\end{lemma}
\begin{proof}
 Cf. \cite[Lemma 12]{Nom}.
\end{proof}

Let $\varphi$ be a (local) isometry of $(M,H,g)$. Take $q_0\in M$ and set $q_1 = \varphi(q_0)$. Then we have an induced transformation which we will denote by $\varphi_*\colon \is(q_0)\longrightarrow\mathfrak{a}(q_1)$ and which acts as follows. To every $(X,A,c)\in\is(q_0)$ we assign an element $(\bar{X},\bar{A},\bar{c})$ of $\mathfrak{a}(q_1)$: $\bar{X} = d_{q_0}\varphi(X)$, $\bar{A} = d_{q_0}\varphi\circ A\circ (d_{q_0}\varphi)^{-1}$, $\bar{c} = c$ (evidently $\bar{A}$ is skew-symmetric). Take $S = \nabla^iR$ or $S = \nabla^id\alpha$. Clearly $\varphi_*S = S$, where $\varphi_*$ stands for the induced action on tensors. By definition, $(\nabla_{X + c\xi(q_0)} + A)S = 0$, and after computations
\begin{gather*}
 (\nabla_{\bar{X} + \bar{c}\xi(q_1)} + \bar{A})S = \varphi_*\big((\nabla_{X + c\xi(q_0)} + A)S \big) = 0\text{.}
\end{gather*}
In this way we have proved
\begin{proposition}\label{prop5}
If $\varphi$ is a local isometry, then
 \begin{align*}
  \varphi_*\colon \is(q)\longrightarrow\is(\varphi(q))
 \end{align*}
is a linear isomorphism for every $q$ belonging to the domain of $\varphi$.
\end{proposition}
Let us go further. Take a smooth horizontal curve $\gamma\colon[0,1]\longrightarrow M$ contained in the domain of $\varphi$, $\gamma(0) = q_0$. Take $(X_0,A_0,c_0)\in\is(q_0)$. Construct two families $(X(t),A(t),c(t))$ and $(\bar{X}(t),\bar{A}(t),\bar{c}(t))$: $(X(t),A(t),c(t))$ is the solution of (\ref{rowXAc}) along $\gamma$ with initial condition $(X_0,A_0,c_0)$ at $t=0$, and $(\bar{X}(t),\bar{A}(t),\bar{c}(t))$ is the solution of (\ref{rowXAc}) along $\varphi\circ\gamma$ with initial condition $\varphi_*(X_0,A_0,c_0)$ at $t=0$. Then, by properties of isometries, we have 
\begin{align}\label{IsomPreser}
 \varphi_*(X(t),A(t),c(t)) = (\bar{X}(t),\bar{A}(t),\bar{c}(t)) \text{,}
\end{align}
$t\in [0,1]$.
\smallskip

Fix a point $q_0$ and its suitably small neighborhood $U$. Denote by $\delta$ the trajectory of $\xi$ such that $\delta(0) = q_0$. Using \cite{AgrIsomp} or \cite{GroSinger} there are coordinates $x^1,\dots,x^{2n+1}$ on $U$ centered at $q_0$ such that 
\begin{align*}
 \xi_{|\delta}  = \frac{\partial}{\partial x^{2n+1}}, \; H_{|\delta} = \Sp\Big\{\frac{\partial}{\partial x^1},\dots,\frac{\partial}{\partial x^{2n}}\Big\} \text{,}
\end{align*}
and $\frac{\partial}{\partial x^1},\dots,\frac{\partial}{\partial x^{2n}}$ from an orthonormal basis for $H$ at points belonging to $\delta$. For a point $q\in U$ denote by $\gamma_q:[0,1]\longrightarrow U$ the curve corresponding in our coordinates to the curve $t\longrightarrow (tx^1(q),\dots,tx^{2n}(q),x^{2n+1}(q))$. Clearly, $\gamma_q(0)\in\delta$, and it can be proved that $\gamma_q$ is a unique minimizing sub-Riemannian geodesic joining $q$ to $\delta$ - see \cite{GroSinger}. As a consequence, $\varphi^s(\gamma_q(t)) = \gamma_{\varphi^s(q)}(t)$, where $\varphi^s$ is the local (flow) of $\xi$. 

Now, starting from an element $(X_0,A_0,c_0)\in\is(q_0)$ we build a field $U\ni q\longrightarrow (X(q),A(q),c(q))\in\is(q)$ smoothly depending on $q$. Namely, in order to obtain $(X(q),A(q),c(q))\in\is(q)$ we solve equations (\ref{rowXAc}) along $\gamma_q$ with initial conditions at $t=0$ equal to $\varphi^s_*(X_0,A_0,c_0)$, where $s$ is such that $\delta(s) = \gamma_q(0)$. If $(X_q(t),A_q(t),c_q(t))$ is the solution, then we set $(X(q),A(q),c(q)) = (X_q(1),A_q(1),c_q(1))$. Using (\ref{IsomPreser}) and what we have said above we see that $[\xi,X] = 0$ and $\mathcal{L}_{\xi}A = 0$. \smallskip

Next we will prove that the equations
\begin{align}\label{EqsY}
 \begin{array}{l}
  \nabla_YX = -AY \smallskip\\
  \nabla_YA = R(X,Y) \smallskip\\
  \nabla_Yc = -d\alpha(X, Y)
 \end{array}
\end{align}
hold on $U$ for every $Y\in\Sec(U,H)$. It is enough to prove (\ref{EqsY}) for every vector $Y\in H_q$, $q\in U$. To this end fix $q\in U$ and $Y\in H_q$. Extend $\dot{\gamma}_q(t)$, $t\in[0,1]$, to a non-vanishing vector field $V\in\Sec(U,H)$. Then, in view of Lemma \ref{lem5}, we can extend $Y$ to a vector field $Y\in\Sec(U,H)$ (denoted by the same letter) such that $P[V,Y] = 0$. In particular, the system of equations
 \begin{align}\label{rowLemm}
 \begin{array}{l}
  \nabla_VX = -AV \smallskip\\
  \nabla_VA = R(X,V) \smallskip\\
  \nabla_Vc = -d\alpha(X, V)
 \end{array}
\end{align}
holds along $\gamma_q$. Now we will prove two lemmas (cf. \cite{Nom}).  

\begin{lemma}\label{LemEq1}
 \[
 \nabla_V(\nabla_YX + AY) = (R(X,Y) - \nabla_YA)V
 \]
holds along $\gamma_q$.
\end{lemma}
\begin{proof}
  We have along $\gamma_q$:
 \begin{gather*}
  \nabla_V(\nabla_YX + AY) =  \nabla_Y\nabla_VX + \nabla_{[V,Y]}X + R(V,Y)X + (\nabla_VA)Y + A(\nabla_VY) = \\
  -(\nabla_YA)V - A(\nabla_YV) + \nabla_{[V,Y]}X + R(V,Y)X + R(X,V)Y + A(\nabla_VY) = \\
  -(\nabla_YA)V + R(X,Y)V + A(\nabla_VY - \nabla_YV) + \nabla_{[V,Y]}X = (R(X,Y) - \nabla_YA)V\text{.}
 \end{gather*}
We used the first Bianchi identity and  $\nabla_VY - \nabla_YV = P[V,Y] = 0$. This last equality implies $[V,Y] = f\xi$ for a smooth function $f$, therefore $\nabla_{[V,Y]}X = f[\xi,X] = 0$.
\end{proof}

\begin{lemma}\label{LemEq2}
 \[
 \nabla_V\big(R(X,Y) - \nabla_YA\big) = - R(\nabla_YX +AY,V)
 \]
holds along $\gamma_q$.
\end{lemma}
\begin{proof}
In the proof one uses the formula $\nabla_V\nabla_Y = \nabla_Y\nabla_V + \nabla_{[V,Y]} + R(V,Y)$ where each term is considered as a derivation. Similar computations as in \cite{Nom} give
\begin{gather*}
 \nabla_V\big(R(X,Y) - \nabla_YA\big) = \big((\nabla_X + A)R\big)(V,Y) - R(\nabla_YX + AY,V) - \nabla_{[V,Y]}A\text{.}
\end{gather*}
The first term on the right-hand side vanishes by definition of the sets $\is(q)$. Moreover, as above, $P[V,Y] = 0$ yields $[V,Y] = f\xi$ for a smooth function $f$, and consequently $\nabla_{[V,Y]}A = f\nabla_{\xi}A = f\mathcal{L}_{\xi}A = 0$ which ends the proof.
\end{proof}
Now we can deduce the first two equations in (\ref{EqsY}). We regard equations from Lemmas \ref{LemEq1}, \ref{LemEq2} as the system of linear ordinary differential equations
\begin{align}
 \begin{array}{l}
  \nabla_V(\nabla_YX + AY) = (R(X,Y) - \nabla_YA)V \\
  \nabla_V\big(R(X,Y) - \nabla_YA\big) = -R(\nabla_YX +AY,V)
 \end{array}
\end{align}
along $\gamma_q$ for the quantities $\nabla_YX + AY$ and $R(X,Y) - \nabla_YA$ with initial conditions 
\begin{align}\label{InitCond}
 \begin{array}{l}
  (\nabla_YX + AY)_{\gamma_q(0)} = 0 \\
  (R(X,Y) - \nabla_YA)_{\gamma_q(0)} = 0
 \end{array} \text{.}
\end{align}
Let us remark that since $Y(\gamma_q(0))$ is a scalar multiple of $\dot{\gamma}_p(0)$ for a certain $p\in U\backslash\delta$ (which is evident from the construction of coordinates $x^1,\dots,x^{2n+1}$ and the curves $\gamma_p$), (\ref{InitCond}) is indeed true and follows directly from the definition of the field $(X,A,c)$. Consequently, $\nabla_YX + AY$ and $R(X,Y) - \nabla_YA$ vanish along $\gamma_q$, in particular at $q$.



So as to prove the third equation in (\ref{EqsY}) we will first derive the following formula
\begin{align}\label{Bianchidalfa}
 (\nabla_Xd\alpha)(Y,Z) + (\nabla_Yd\alpha)(Z,X) + (\nabla_Zd\alpha)(X,Y) = 0
\end{align}
which holds for every $X,Y,Z \in \Sec(H)$. In fact, it is enough to write out the expression for $d^2\alpha = 0$, notice that e.g. $d\alpha(X,[Y,Z]) = d\alpha(X,P[Y,Z])$ etc., and use the fact that the horizontal torsion vanishes -- cf. Section \ref{sRCON}.

\begin{lemma}\label{LemEq3}
 \[
 \nabla_V\big(\nabla_Yc + d\alpha(X,Y)\big) = 0
 \]
holds along $\gamma_q$.
\end{lemma}
\begin{proof}
 \begin{gather*}
  \nabla_V\big(\nabla_Yc + d\alpha(X,Y)\big) = \nabla_Y\nabla_Vc + \nabla_{[V,Y]}c + \nabla_V(d\alpha(X,Y)) = \\
  - \nabla_Y(d\alpha(X,V))+ \nabla_V(d\alpha(X,Y)) = (\nabla_Xd\alpha)(V,Y) - d\alpha(AV,Y) - d\alpha(V,AY) = \\
  ((\nabla_X + A)d\alpha)(V,Y) = ((\nabla_{X + c\xi} + A)d\alpha)(V,Y) = 0\text{,}
 \end{gather*}
 where the last term vanishes by definition of the sets $\is(q)$. We also used (\ref{Bianchidalfa}), the first equation in (\ref{EqsY}) and the fact that the derivation induced by $A$ is zero on functions. Moreover, as above, $[V,Y] = f\xi$ for a smooth function $f$ which gives $\nabla_{[V,Y]}c = f\xi(c)$ and which is zero by construction of $c$.
\end{proof}
Now, similarly as above, we deduce the last equation in (\ref{EqsY}). Namely, we view the expression from Lemma \ref{LemEq3} as the system of ordinary differential equation along $\gamma_q$ for the unknown quantity $\nabla_Yc + d\alpha(X,Y)$ and with initial condition $(\nabla_Yc + d\alpha(X,Y))_{\gamma_q(0)} = 0$ (this is indeed zero by the same reasoning as above).

The final step i proving Theorem \ref{th2} is the proposition below.
\begin{proposition}\label{FinalStep}
 Suppose that equations (\ref{EqsY}) are satisfied in $U$ for every $Y\in\Sec(U,H)$. Then $Z = X + c\xi$ is an infinitesimal isometry.
\end{proposition}
\begin{proof}
 Fix $W_1,W_2\in\Sec(H)$. Since $A$ is skew-symmetric we have
\begin{gather*}
 0 = -g(AW_1,W_2) - g(W_1,AW_2) = g(\nabla_{W_1}X,W_2) + g(W_1,\nabla_{W_2}X) = \\
 g(\nabla_{X}W_1,W_2) + g(W_1,\nabla_XW_2) + g(P[W_1,X],W_2) + g(W_1,P[W_2,X]) = \\
 X(g(W_1,W_2)) - g(P[X,W_1],W_2) - g(W_1,P[X,W_2])\text{.}
\end{gather*}
Using this and remembering that $\xi$ is an infinitesimal isometry we further obtain
\begin{gather*}
 (X + c\xi)(g(W_1,W_2)) = g(P[X,W_1],W_2) + g(W_1,P[X,W_2]) + c\xi(g(W_1,W_2)) = \\
 g(P[X,W_1],W_2) + g(W_1,P[X,W_2]) + g(c[\xi,W_1],W_2) + g(W_1,c[\xi,W_2]) = \\
 g(P([X,W_1] + c[\xi,W_1]),W_2) + g(W_1,P([X,W_2] + c[\xi,W_2]))\text{.}
\end{gather*}
Next, using $[c\xi,W_i] = -W_i(c)\xi + c[\xi,W_i]$ we are led to
\begin{gather*}
 (X + c\xi)(g(W_1,W_2)) = g(P[X + c\xi,W_1],W_2) + g(W_1,P[X + c\xi,W_2]) \text{,}
\end{gather*}
i.e., 
\begin{gather*}
 Z(g(W_1,W_2)) = g(P[Z,W_1],W_2) + g(W_1,P[Z,W_2]) \text{.}
\end{gather*}
Finally, since $c = \alpha(Z)$, the last equation in (\ref{EqsY}) reads $Y(\alpha(Z)) = -d\alpha(X,Y) = -d\alpha(Z,Y)$ for every $Y\in\Sec(U,H)$ which, by Proposition \ref{ContCond}, shows that $Z$ is a contact vector field. Summing up all what we have said we arrive at
\begin{gather*}
 Z(g(W_1,W_2)) = g([Z,W_1],W_2) + g(W_1,[Z,W_2])
\end{gather*}
which ends the proof.
\end{proof}
In this way Theorem \ref{th2} is proven.

\subsection{Theorem \ref{th4} and Corollary \ref{corIntr}}
Theorem \ref{th4} is now clear just like Corollary \ref{corIntr}. We only mention here how Corollary \ref{corIntr} follows from the Riemannian case considered in \cite{Nom}.  

Let $(M,H,g)$ is a contact sub-Riemannian manifold, $\dim M = 2n+1$. As above we suppose it to be oriented as a contact manifold. As a result the Reeb vector field $\xi$ (which for $n$ even is defined up to sign) is globally defined. Now it is a standard thing, which is done in many articles, to extend the metric $g$ to the whole tangent bundle $TM$ in the following way. We define the Riemannian metric $\tilde{g}$ by formula: 
\begin{align*}
 \tilde{g}_{|H\times H} = g, \; \tilde{g}(\xi,\xi) = 1, \; \tilde{g}_{|H\times \Sp\{\xi\}} = 0 \text{.}
\end{align*}
Suppose that $Z$ is an infinitesimal isometry of $(M,H,g)$. Using Proposition \ref{InfIsmCon}, \ref{PropIsom} and the definition of $\tilde{g}$ it is easy to see that $Z$ is an infinitesimal isometry of the Riemannian manifold $(M,\tilde{g})$. Moreover it can be seen that an infinitesimal isometry of $(M,\tilde{g})$, which is a contact vector field, is an infinitesimal isometry of $(M,H,g)$.

Suppose now that $(M,H,g)$ is an oriented analytic simply connected sub-Riemannian manifold (here we do not have to assume that $(M,H,g)$ is special). Surely, $\tilde{g}$ is then also analytic. Let $q\in M$, $Z^*\in\is^*(q)$, and suppose that $Z$ is a representative of $Z^*$. Clearly, $Z$ is analytic by the above remark and by \cite{Palais}. Thus $Z$ is a local infinitesimal isometry of $(M,\tilde{g})$ and by \cite[Theorem 1]{Nom} it can be extended to the globally defined infinitesimal isometry of $(M,\tilde{g})$ also denoted by $Z$. Now, Proposition \ref{PropIsom} insures that $Z$ is an infinitesimal isometry of $(M,H,g)$.


\bigskip



\begin{thebibliography}{99}
\bibitem{Agr}A.A. Agrachev, \textit{Exponential mappings for contact sub-Riemannian structures}. J. Dynam. Control Systems 2 (1996), no. 3, 321--358.
\bibitem{AgrIsomp}A.A. Agrachev, J.P.A. Gauthier \textit{sub-Riemannian Metrics and Isoperimetric Problems in the Contact Case}. J. Mathematical Sciences vol. 103, no. 6, 2001.
\bibitem{AgrBar}A.A. Agrachev, D. Barilari \textit{Sub-Riemannian structures on 3D Lie groups}. J. Dyn. Control Syst. 18 (2012), no. 1, 21--44.
\bibitem{BoscAgr}A. A. Agrachev, D. Barilari, U.Boscain, \textit{Introduction to Riemannian and Sub-Riemannian geometry}. Preprint SISSA 09/2012/M.
\bibitem{AMS} D. Alekseevsky, A. Medvedev, J. Slovak, \textit{Constant Curvature Models in Sub-Riemannian Geometry},  J. Geom. Phys. 138 (2019), 241--256.
\bibitem{POP} D. Barilari, L. Rizzi, \text{A formula for Popp's volume in sub-Riemannian geometry}. Anal. Geom. Metr. Spaces 1 (2013), 42--57.
\bibitem{Bell} A. Bella\"{i}che, \textit{The tangent space in sub-Riemannian geometry}, Dynamical systems, 3. J. Math. Sci. (New York) 83 (1997), no. 4, 461--476.
\bibitem{Blair} David E. Blair, \textit{Riemannian geometry of contact and symplectic manifolds}. Second edition. Progress in Mathematics, 203. Birkhäuser Boston, Ltd., Boston, MA, 2010. xvi+343 pp. ISBN: 978-0-8176-4958-6.
\bibitem{leDonne}L. Capogna, E. Le Donne, \textit{Smoothness of sub-Riemannian isometries}. Amer. J. Math. 138 (2016), no. 5, 1439--1454.
\bibitem{Gorod} E. Falbel, C. Gorodski, M. Rumin, \textit{Holonomy of sub-Riemannian manifolds}. Internat. J. Math. 8 (1997), no. 3, 317--344.
\bibitem{GrKrESAIM} M. Grochowski, W. Kry\'nski, \textit{On contact sub-pseudo-Riemannian isometries}. ESAIM Control Optim. Calc. Var. 23 (2017), no. 4, 1751--1765.  
\bibitem{GrKrRicam} Grochowski, Marek; Kry\'nski, Wojciech, \textit{Invariants of contact sub-pseudo-Riemannian structures and Einstein-Weyl geometry}. Variational methods, 434--453, Radon Ser. Comput. Appl. Math., 18, De Gruyter, Berlin, 2017.
\bibitem{GroWar} M. Grochowski, B. Warhurst, \textit{Isometries of sub-Riemannian metrics supported on Martinet type distributions}, J. Lie Theory 28 (2018), No. 3, 767—780. 
\bibitem{GroConn} M. Grochowski, \textit{Connections on bundles of horizontal frames associated with contact sub-pseudo-Riemannian manifolds}, J. Geom. Phys. 146 (2019), 103518, 13 pp.
\bibitem{GroConnCor} M. Grochowski, \textit{Corrigendum to “Connections on bundles of horizontal frames associated with contact sub-pseudo-Riemannian manifolds” [J. Geom. Phys. 146 (2019) 103518]},  J. Geom. Phys. 178 (2022), Paper No. 104582.
\bibitem{GroSinger} M. Grochowski, \textit{A de Rham decomposition type theorem for contact sub-Riemannian manifolds}, Anal. Math. Phys. 12 (2022), no. 1, Paper No. 13.
\bibitem{GroAnn} M. Grochowski, \textit{On the dimension of the algebras of local infinitesimal isometries of 3-dimensional special sub-Riemannian manifolds}, Ann. Mat. Pura Appl. (4) 203 (2024), no. 1, 385--401.
\bibitem{GroMaVas}E. Grong, I. Markina, A. Vasil'ev, \textit{Sub-Riemannian geometry on infinite-dimensional manifolds}. J. Geom. Anal. 25 (2015), no. 4, 2474--2515.
\bibitem{KoNo} S. Kobayashi, K. Nomizu, \textit{Foundations of differential geometry. Vol. I.} Reprint of the 1963 original. Wiley Classics Library. A Wiley-Interscience Publication. John Wiley \& Sons, Inc., New York, 1996. {\rm xii}+329 pp. ISBN: 0-471-15733-3
\bibitem{Kostant} B. Kostant, \textit{Holonomy and the Lie algebra of infinitesimal motions of a Riemannian manifold}, Trans. Amer. Math. Soc. 80 (1955), 528--542.
\bibitem{LuiSuss} W. Liu, H.J. Sussman, \textit{Shortest paths for sub-Riemannian metrics on rank-two distributions}. Mem. Amer. Math. Soc. 118 (1995), no. 564.
\bibitem{Nom} K. Nomizu, \textit{On local and global existence of Killing vector fields}, Ann. of Math. (2) 72 (1960), 105--120. 
\bibitem{Palais} R.S. Palais, \textit{On the differentiability of isometries}, Proc. Amer. Math. Soc. 8 (1957), 805–807.
\bibitem{Strich} R.S Strichartz, \textit{Sub-Riemannian geometry}, J. Differential Geom. 24 (1986), no. 2, 221--263.




\end{thebibliography}
\end{document}